\documentclass[11pt]{amsart}
\usepackage{amstext}
\usepackage[english]{babel}
\usepackage{hyperref}
\usepackage{a4wide,color,graphicx,amsmath,amssymb,verbatim,mathrsfs,latexsym}
\allowdisplaybreaks
\newtheorem{theo}{Theorem}

\newtheorem{lemma}{Lemma}

\newcommand{\R}{\mathbb{R}}	
\newcommand{\N}{\mathbb{N}}	
\newcommand{\eps}{\varepsilon}	
\newcommand{\pa}{\partial}		
\newcommand{\Div}{\textrm{div}\,}	
\newcommand{\Curl}{\textrm{curl}\,}	

\newcommand{\na}{\nabla}		
\newcommand{\OmT}{			
\Omega\times (0,T) }	
\newcommand{\uml}{\"}
\newcommand{\IRd}{\int_{\R^2}}

\newcommand{\black}[1]{\textcolor{black}{#1}}

\newcommand{\utau}{u^{(\tau)}}
\newcommand{\ptau}{p^{(\tau)}}
\newcommand{\urho}{u^{(\rho)}}
\newcommand{\prho}{p^{(\rho)}}
\newcommand{\QT}{\R^2\times (0,T)}
\newcommand{\II}[1]{\int_0^T\int_{\R^2} #1\, dx dt}
\newcommand{\urhop}{u^{(\rho')}}

\title[Existence of weak solutions to a continuity equation with space time nonlocal  Darcy law ]{{{Existence of weak solutions to a continuity equation with space time nonlocal Darcy law }}}
\author{Luis Caffarelli, Maria Gualdani, Nicola Zamponi}
\address{The University of Texas at Austin
Mathematics Department RLM 8.100
2515 Speedway Stop C1200
Austin, Texas 78712-1202}
\email{caffarel@math.utexas.edu}
\address{The University of Texas at Austin
Mathematics Department RLM 8.100
2515 Speedway Stop C1200
Austin, Texas 78712-1202}
\email{gualdani@math.utexas.edu }
\address{University of Mannheim, School of Business Informatics and Mathematics, B6, 28, 68159 Mannheim (Germany)}
\email{nzamponi@mail.uni-mannheim.de}

\date{\today}

\begin{document}

\thanks{ LC is supported by NSF DMS-1540162. MPG is supported by DMS-1514761 and would like to
thank NCTS Mathematics Division Taipei for their kind hospitality. NZ acknowledges partial support from the Austrian Science Fund (FWF), grants P22108, P24304, W1245 and from the Czech Science Foundation, project No. 19-04243S, as well as support from the Alexander von Humboldt foundation.} 

\begin{abstract}
In this manuscript we consider a non-local porous medium equation with non-local diffusion effects given by a fractional heat operator
\begin{equation*}
    \begin{cases}
      \pa_t u = \Div(u\na p),\\
      \pa_t p = -(-\Delta)^s p + u^\beta,
    \end{cases}
\end{equation*}
 in two space dimensions for $\beta>1$, $\frac{1}{\beta}<s<1$. Global in time existence of weak solutions is shown by employing a time semi-discretization of the equations, an energy inequality and the Div-Curl lemma. 
\end{abstract}
%

\maketitle


\section{Introduction}

\black
In this manuscript we study existence of weak solutions to a porous medium equation with non-local diffusion effects: 
\begin{equation}\label{1L}
    \begin{cases}
      \pa_t u = \Div(u\na p),\\
      \pa_t p = -(-\Delta)^s p + u^\beta.
    \end{cases}
\end{equation}
Here $u(x,t)\ge 0$ denotes the density function and $p(x,t)\ge 0$ the pressure. We analyze the problem when $x\in \R^2$, $ \frac{1}{\beta} < s < 1$ and $\beta>1$. The model describes the time evolution of a density function $u$ that evolves under the continuity equation 
$$
\pa_t u = \Div(u {\bf{v}}),
$$
where the velocity is conservative, ${\bf{v}}=\nabla p$, and $p$ is related to $u^\beta$ by the inverse of the fractional heat operator $\pa_t + (-\Delta )^s$.

Problem (\ref{1L}) is the parabolic-parabolic version of a parabolic-elliptic problem recently studied in \cite{BIK15}. In \cite{BIK15}, the authors proved the existence of sign-changing weak solutions to 
\begin{align}\label{BIKeq} 
\pa_t u = \Div(|u|\nabla ^{\alpha-1}(|u|^{m-2}u)).
\end{align}
For $m=q+1$ and $\alpha = 2-2s$ equation (\ref{BIKeq}) reads as 
\begin{align*}
 \pa_t u & = \Div(u\na p), \quad p = (-\Delta )^{-s} u^\beta ,\quad 0<s\le 1.
\end{align*}
The presence of $\pa_t p$ makes our system quite different from  (\ref{BIKeq}). For example, techniques such as maximum principle and Stroock-Varopoulos inequality do not work. We overcome these significant shortcomings with the introduction of ad-hoc regularization terms, together with suitable compact embeddings and moment estimates. See later for a more detailed explanation.

A {\em{linear}} parabolic-elliptic version of (\ref{1L})
\begin{align}
 \pa_t u & = \Div(u\na p), \quad p = (-\Delta )^{-s} u ,\quad 0<s\le 1, \label{CV} 
\end{align}
was studied by the first author and collaborators in a series of papers: existence of weak solutions for (\ref{CV}) is proven in \cite{CV11, CV15, SV14} and H\"older regularity in  \cite{CSV13}. The case $s=1$ also appeared in \cite{AS08} as a model for superconductivity.     

Systems (\ref{CV}) and (\ref{1L}) are reminiscent to a well-studied macroscopic model proposed for phase segregation in particle systems with long range interaction:   
\begin{equation}\label{1I}
    \begin{cases} \pa_t u  =  \Delta u + \Div(\sigma(u)\na p),\\
 p  = K \ast u.
  \end{cases}
\end{equation}
Any system that exhibits coexistence of different densities (for example, fluid and vapor or fluid and solid) has equilibrium configurations that segregate into different regions; the surface of these regions are minimizers of a {free energy functional}. The relaxation to equilibrium of the density function $u(x,t)$ can be described in general by nonlinear integro-differential equations of type (\ref{1I}).  One example is the model proposed in \cite{GL97}, in which the mobility is $\sigma(u) := u(1-u)$ and the kernel $K$ is bounded, symmetric and compactly supported. Such model describs the hydrodynamic (or mean-field) limit of a microscopic model undergoing phase segregation with particles interacting under a short-range and long-range Kac potential.  Several other variants of (\ref{1I}) are present in the literature \cite{R91, GL97, GLM00, GL98, STV18}. We also mention \cite{LMG01} for the study of a deterministic particle method for heat and Fokker–Planck equations of porous media type where the non-locality appears in the coefficients. The long time behavior of weak solutions to \eqref{1L} was studied in \cite{DGZ19}. There the authors show algebraic decay in time towards the stationary solutions $u=0$ and $\nabla p =0$. \\

The condition that the pressure satisfies a parabolic equation introduces non-trivial complications in the analysis of (\ref{1L}). The non-local structure prevents the equation from having  a comparison principle. Moreover, maximum principle does not give useful insights, since at any point of maximum for $u$ we only know that $\pa_t u \le u \Delta p$.  We overcome the lack of comparison and maximum principles with the introduction of several regularizations. Stampacchia's truncation arguments yield non-negativity of the solutions and the Div-Curl lemma will be used to identify the limit for $u^\beta$. 

The main result of this manuscript is summarized in the following theorem: 
\begin{theo}\label{main_thm} 
Let $\beta>1$, $\frac{1}{\beta}<s<1$.
Moreover let $u_{in}, p_{in} : \R^2 \to [0,+\infty)$ be functions such that 
$u_{in}\in L^1\cap L^\beta(\R^2)$, $p_{in}\in L^1\cap H^1(\R^2)$. 
There exist functions $u, p: \R^2 \times [0,{\infty)} \to [0,+\infty) $ such that for every $T>0$ 
\begin{align*}
& u\in L^\infty(0,T,L^1\cap L^\beta(\R^2))\cap
L^{\beta+1}(\QT), \\ 
& p\in L^\infty(0,T,H^{1}\cap L^1(\R^2))\cap L^2(0,T,H^{s+1}(\R^2)), \\
& \pa_t u  \in L^{\beta+1}(0,T, W^{-1,\frac{2(\beta+1)}{\beta+3}}(\R^2)), \quad 
\pa_t p\in L^{(\beta+1)/\beta}(0,T,(L^2\cap L^{\beta+1}(\R^2))'),
\end{align*}
which satisfy the following weak formulation to (\ref{1L}): 
\begin{align}\label{weak.u}
& \int_0^T  \langle \pa_t u, \phi\rangle dt 
+\int_0^T\int_{\R^{2}}u\nabla p\cdot\nabla\phi\, dx dt = 0\quad \forall \phi \in L^{\frac{\beta+1}{\beta}}(0,T; W^{1,\frac{2(\beta+1)}{\beta-1}}(\R^2)), \\ 
\label{weak.p}
&\int_0^T  \langle \pa_t p, \psi\rangle  dt  + \int_0^T\int_{\R^2}( (-\Delta)^{s} p - {u^\beta}) \psi dx dt= 0
\quad\forall \psi \in {L^{\beta+1}(0,T; L^2\cap L^{\beta+1}(\R^2))},\\
\nonumber
&\lim_{t\to 0} u(t) = u_{in}\quad\textrm{in} \; W^{-1,\frac{2(\beta+1)}{\beta+3}}(\R^2),\quad  \lim_{t\to 0} p(t) = p_{in} \quad\textrm{in}\; (L^2\cap L^{\beta+1}(\R^2))'  ,
\end{align}
as well as the mass conservation relation
$$
\int_{\R^2}u(x,t)dx = \int_{\R^2}u_{in}(x)dx , \quad t>0.
$$
\end{theo}

The starting point about our analysis is the observation that 
$$ H[u,p] := \IRd\left( \frac{u^\beta}{\beta-1} + \frac{1}{2}|\na p|^2 \right)dx  $$
is a Lyapunov functional for (\ref{1L}) and satisfies the bound
$$
H[u,p] + \int_0^T \IRd |(-\Delta)^{s/2}\na p |^2 dx dt= H[u_{in},p_{in}]. 
$$
Indeed, formal computations show that 
\begin{align*}
\frac{d}{dt}\IRd \frac{u^\beta}{\beta-1} dx &= \left\langle \Div(u\na p) , 
\frac{\beta u^{\beta-1}}{\beta-1}\right\rangle = -\IRd \na u^{\beta}\cdot\na p dx .
\end{align*}
Testing the equation for $p$ against $\Delta p$ we obtain
\begin{align*}
\IRd \na u^{\beta}\cdot\na p dx &
= \frac{d}{dt}\IRd\frac{|\na p|^2}{2}dx + \IRd |(-\Delta)^{s/2}\na p |^2 dx ,
\end{align*}
which leads to 
\begin{equation}\label{H_ineq}
\frac{d}{dt}H[u,p] + \IRd |(-\Delta)^{s/2}\na p |^2 dx = 0\qquad t>0.
\end{equation}


The major difficulty, in the approximation process, is the identification of the limit of $u^\beta$. The energy inequality (\ref{H_ineq}) provides plenty of informations for the pressure $p$, but only uniform integrability in $L^\infty(L^\beta)$ for $u$. At the moment it is unclear to the authors how to use the bounds for $\nabla p$ to get useful bounds for $\nabla u$ or $u$. To overcome the lack of compactness we employ the Div-Curl Lemma (see \cite{GasMar08}) to the vector fields 
$$
U^\eps \equiv (u^\eps , -u^\eps\nabla p^\eps),\qquad
V^\eps \equiv (\pa_t p^\eps , \na p^\eps) ,
$$
where $(u^\eps,p^\eps)$ is a suitable approximate solution to \eqref{1L}. The argument yields 
$$
U^\eps\cdot V^\eps\rightharpoonup U\cdot V \qquad\mbox{weakly in }L^1(\R^2\times (0,T)),
$$
where $U$, $V$ are the weak limits of $U^\eps$, $V^\eps$, respectively. Strong convergence of $p^\eps$ and standard result in compensated compactness theory \cite{FeiNov} yield strong convergence for $u^\eps$.

{{The application of the Div-Curl Lemma brings two restrictions on the system. The first one concerns the lower bound for $s$, $s>\frac{1}{\beta}$, the second one the dimension. It is unclear how to remove such restrictions, as they seem necessary to fulfill the integrability and compactness constraints on the quantities $U^\eps$, $V^\eps$. The assumptions on $s$, $\beta$ and $d$ are not satisfactory from the point of view of a general theory for weak solutions. As such, Theorem 1 is a first step to understand the complete behavior of (1). Most interesting however, is the fact that the addition of a nonstationary term in the pressure equation radically changes the behavior of the system and calls for a different analytical setting than in \cite{CV11, STV18}. We also point out that the successful use of the Div-Curl Lemma, a tool commonly employed in the study of fluid-dynamic systems, in the analysis of nonlocal diffusion equations is (to our best knowledge) a novelty and an unexpected connection between the two fields.
Uniqueness of weak solutions is an important open question for our system. We expect it to hold for short time straightforwardly. For long time the only available result so far is the one in \cite{DGZ19}, in which the authors show a weak-strong uniqueness result: if there exists a strong solution, then any weak solution with the same initial data coincides with it.\\
Existence of a solution for $\beta =1$ appears to be out of reach with the present technique, as several other terms will lack compactness.
  }}

The paper is organized as follows: in Section \ref{tech_res}, we show two preliminary technical lemmas, and in Section \ref{main_sec} the proof of the main theorem. 



\section{Some technical results} \label{tech_res}
\begin{lemma}\label{lem_comp}
Let $g : [0,\infty)\to [0,\infty)$ be a continuous, nondecreasing function such that $\lim_{r\to\infty}g(r)=\infty$.
For $\kappa\in (0,p]$, $1\leq p < 2$ define the functional space $V_{g,\kappa,p}$ as 
$$
V_{g,\kappa,p} := W^{1,p}(\R^2)\cap L^\kappa(\R^2,g(|x|)dx) = 
\left\{ f\in W^{1,p}(\R^2) \; : \; \int_{\R^2} f(x)^\kappa g(|x|)dx < \infty\right\} .
$$ 
Then $V_{g,\kappa,p}$ is compactly embedded in $L^{q} (\R^2)$ for any $\max\{\kappa,1\}\leq q < \frac{2p}{2-p} $.
\end{lemma}
\begin{proof}
Let $\{f_n\}$ be a uniformly bounded sequence in $V_{g,\kappa,p}$. We first notice that there exists a subsequence, still denoted with $f_n$ such that 
$$
f_n \rightharpoonup f \; \textrm{weakly in } \; W^{1,p}(\R^2) \hookrightarrow L^{2p/(2-p)}(\R^2).
$$
Denote with $B_R$ the ball of center $x=0$ and radius $R$. Since $W^{1,p}(B_R)$ is compactly embedded in $L^q (B_R)$ for any $1\le q < \frac{2p}{2-p}$, there exists a subsequence of $f_n$, still denoted with $f_n$, such that $$
f_n \to f\; \textrm{strongly in }  L^q (B_R)\quad \textrm{for any} \;1\le q < \frac{2p}{2-p.}$$

 Thanks to a Cantor diagonal argument, the subsequence $f_n$ can be chosen to be independent of $R$. The uniform bound for $f_n$ in $V_{g,\kappa,p}$ and Fatou's Lemma imply that $f\in V_{g,\kappa,p}$.


Next we show that $|f_n - f|^\kappa$ strongly in $L^1(\R^2)$: for $n$ big enough 
\begin{align*}
\int_{\R^2} | f_n-f|^\kappa\;dx =& \int_{B_R} | f_n-f|^\kappa\;dx + \int_{B_R^c} | f_n-f|^\kappa\;dx\\
 \le & \;\frac{\varepsilon}{2} + \frac{1}{g(R)}\int_{B_R^c} g(|x|) | f_n-f|^\kappa\;dx \le \varepsilon,
\end{align*}
by choosing $R$ big enough. Interpolation between $L^{2p/(2-p)}$ and $L^{\max\{\kappa,1\}}$ implies that for any $q$ with $\max\{\kappa,1\}\leq q < \frac{2p}{2-p} $ the sequence $f_n$ strongly converges to $f$ in $L^q(\R^2)$.

\end{proof}

\begin{lemma}\label{frac_lapl_eta_R}
Define $\eta(x) = (1+|x|^2)^{-\alpha/2}$ with $\alpha > 4$ and for every $R\geq 1$ we set $\eta_R(x) = \eta(x/R)$. For $s>0$ we have 
$$
\lim_{R\to\infty} \|(-\Delta)^s \eta_R\|_{L^\infty} = 0.
$$
\end{lemma}
\begin{proof}
{{The result is a consequence of the scaling property of the fractional laplacian:
$$
(-\Delta)^s \eta_R = \frac{1}{R^{2s}}(-\Delta)^s \eta.
$$
}}
\end{proof}

\section{Proof of the main theorem} \label{main_sec}

Define the spaces 
$$ X := L^{\frac{2\beta}{\beta-1}}(\R^2),\qquad 
Y := \left\{ g\in W^{1,\frac{1+\beta}{\beta}}(\R^2)~ : ~ \int_{\R^2}|g|^{\frac{1+\beta}{\beta}}\gamma dx < \infty \right\},$$
$$
\tilde Y \equiv \{ u\in L^1_{loc}(\R^2) ~ : ~~ u\geq 0 ~ \mbox{a.e. in }\R^2, ~~ u^{\beta-1}\in Y \} ,
$$
where 
$$\gamma(x) := \sqrt{1+|x|^2}.$$ 
Thanks to Sobolev's embedding and Lemma \ref{lem_comp}: 
\begin{align}
	\label{Sob.1}
& Y\hookrightarrow L^q(\R^2)\quad\mbox{continuously for }\frac{1+\beta}{\beta}\leq q \leq \frac{2(\beta+1)}{\beta-1},\\
	\label{Sob.2}
& Y\hookrightarrow L^q(\R^2)\quad\mbox{compactly for }\frac{1+\beta}{\beta}\leq q < \frac{2(\beta+1)}{\beta-1}.
\end{align}
In particular, the embedding $Y\hookrightarrow X$ is compact.\\
For every measurable function $g : \R^2\to\R\cup\{\pm\infty\}$ we denote by $g_+ := \max \{g,0\}$ and $g_- := \min \{g,0\}$ its positive and negative part, respectively.\\

For given constants $\black{\black{\varrho_1}}, \black{\black{\varrho_2}}, \tau, \eps > 0$,  functions $u^*\in\tilde Y$ and $p^* \in H^{2s}(\R^2)$ such that $u^* , p^*\geq 0$ a.e. in $\R^2$, consider the time-discrete problem
\begin{align}\label{mar15.linpb.1_a}
\int_{\R^{2}}\left( \frac{u - u^*}{\tau}\phi + u\nabla p\cdot\nabla\phi  + \black{\varrho_1} |\nabla u^{\beta-1}|^{\frac{1}{\beta}-1}\nabla u^{\beta-1}\cdot\nabla\phi  + \eps u^{\frac{\beta-1}{\beta}}\phi\gamma\right) \; dx & = 0\quad \forall \phi\in Y ,\\
\label{mar15.linpb.2_b}
\frac{p - p^*}{\tau}+ (-\Delta)^{s}p - \black{\black{\varrho_2}}\Delta p - u^\beta = 0.
\end{align}

We divide the proof of Theorem \ref{main_thm} into several steps: we first show existence of solution to (\ref{mar15.linpb.1_a}), (\ref{mar15.linpb.2_b}) by Leray-Schauder fixed point theorem. Then we perform the limits $\eps \to 0$, $\tau \to  0$, $\black{\varrho_2}\to 0$ and $\black{\varrho_1} \to 0$ (in this order). The last limit is the most complicated because we need compactness for $u$ without relying on the term $ \black{\varrho_1}\int_{\R^2}|\nabla u^{\beta-1}|^{\frac{1}{\beta}-1}\nabla u^{\beta-1}\cdot\nabla\phi  dx$.

\subsection{Existence for (\ref{mar15.linpb.1_a})-(\ref{mar15.linpb.2_b})} 

For given constants $\varrho, \tau, \eps > 0$, $\sigma \in [0,1]$, functions $z\in X$, $u^*\in \tilde Y$ and $p^* \in H^{2s}(\R^2)$ such that $u^* , p^*\geq 0$ a.e. in $\R^2$, consider the linear problem in the variable $w$:
\begin{align}
& \int_{\R^{2}}( \tau^{-1}(|w|^{\frac{2-\beta}{\beta-1}}w - u^*)\phi +   \sigma z_+^{\frac{1}{\beta-1}}\nabla p\cdot\nabla\phi)dx + \black{\black{\varrho_1}}\int_{\R^2}|\na w|^{\frac{1}{\beta}-1}\nabla w\cdot\nabla\phi dx \nonumber\\ 
&\qquad + \eps\int_{\R^2}|w|^{\frac{1}{\beta}-1}w\phi\gamma dx = 0\quad \forall \phi\in Y , \label{mar15.linpb.1}\\
\label{mar15.linpb.2}
&\int_{\R^2}( \tau^{-1}(p - p^*)\psi + (-\Delta)^{s/2}p\cdot (-\Delta)^{s/2}\psi + \black{\black{\varrho_2}}\nabla p\cdot\nabla\psi - z_+^{\frac{\beta}{\beta-1}}\psi )dx = 0\quad\forall \psi \in H^1(\R^2).
\end{align}

We first solve \eqref{mar15.linpb.2}. We have that $z^{\frac{\beta}{\beta-1}}\in L^2(\R^2)$.  Lax-Milgram Lemma yields the existence of a unique solution
$p\in H^1(\R^2)$. Standard elliptic regularity results imply that $p\in H^{2}(\R^2)$ and consequently $\nabla p\in L^q(\R^2)$ for every $q\geq 2$.

We now solve \eqref{mar15.linpb.1}.  Since $z^{\frac{1}{\beta-1}}\in L^{2\beta}(\R^2)$ and $\nabla p \in L^q(\R^2)$ for every $q\geq 2$, the linear mapping
$$ \phi\in Y\mapsto \int_{\R^2}(-\tau^{-1}u^*\phi + \sigma z_+^{\frac{1}{\beta-1}} \nabla p\cdot\nabla\phi) dx \in \R  $$
is continuous. The nonlinear operator $\mathcal{A}: Y \to Y'$ defined by
\begin{align*}
\langle \mathcal A[w] , \phi \rangle &= 
\int_{\R^{2}}\tau^{-1}|w|^{\frac{2-\beta}{\beta-1}}w \phi dx + \black{\black{\varrho_1}}\int_{\R^2}|\na w|^{\frac{1}{\beta}-1}\nabla w\cdot\nabla\phi dx + \eps\int_{\R^2}|w|^{\frac{1}{\beta}-1}w\phi\gamma dx
\end{align*}
for every $\phi\in Y$ is strictly monotone, coercive, hemicontinuous. Therefore
the standard theory of monotone operators \cite{Zei90} yields the existence of a unique solution $w\in Y$ to \eqref{mar15.linpb.1}. 

We can now define the mapping 
$$F : (z,\sigma)\in X\times [0,1]\mapsto w\in X,
$$ 
where $(w,p) \in Y\times H^2 (\R^2)$ is the unique solution to \eqref{mar15.linpb.1}, \eqref{mar15.linpb.2}.
Clearly $F(\cdot,0)$ is a constant mapping. Moreover $F$ is continuous and also compact due to the compact embedding $Y\hookrightarrow X$, see Lemma \ref{lem_comp}. 

Next, we show that any fixed point is nonnegative and uniformly bounded in $\sigma$. We use a Stampacchia truncation argument. This method is generally used in nonlinear elliptic problems to show positivity, boundedness and higher regularity via the choice of particular test functions. In our case, by choosing $\phi=w_-$ and $\psi = p_-$ as test functions, we get 
\begin{align*}
\int_{\R^2}\tau^{-1}|w_-|^{\frac{\beta}{\beta-1}} dx +
\int_{\R^{2}} \black{\black{\varrho_1}}|\nabla w_-|^{(\beta+1)/\beta} dx
+\eps\int_{\R^2}|w_-|^{(\beta+1)/\beta} \gamma dx &= 0,\\
\int_{\R^2} \tau^{-1}(p_-)^2 + ((-\Delta)^{s/2}p_-)^2 + \black{\black{\varrho_2}}|\nabla p_-|^2dx &\le 0,
\end{align*}
from which it follows that $w, p\geq 0$ a.e.~in $\R^2$. 
The nonnegativity of $w$ and the $H^2(\R^2)$-regularity of $p$ allow for the formulation 
\begin{align}\label{mar15.fp.1}
\int_{\R^{2}}( \tau^{-1}(u - u^*)\phi + \sigma u\nabla p\cdot\nabla\phi)dx &+ \black{\varrho_1} \int_{\R^2}|\nabla u^{\beta-1}|^{1/\beta-1}\nabla u^{\beta-1}\cdot\nabla\phi dx\\ \nonumber 
&+ \eps\int_{\R^2}u^{(\beta-1)/\beta}\phi\gamma dx = 0\quad \forall \phi\in Y ,\\
\label{mar15.fp.2}
 \tau^{-1}(p - p^*) + (-\Delta)^{s}p& - \black{\varrho_2}\Delta p - u^\beta = 0 \qquad\mbox{in }\R^2 ,
\end{align}
where we defined $u\equiv w^{\frac{1}{\beta-1}}$.

We now search for uniform bounds with respect to $\sigma$: choosing $\phi = u^{\beta-1}$ in \eqref{mar15.fp.1} leads to
\begin{align*}
&\int_{\R^{2}}( \tau^{-1}(u - u^*)u^{\beta-1} + \black{\varrho_1} \int_{\R^2}|\nabla u^{\beta-1}|^{(\beta+1)/\beta} dx + \eps\int_{\R^2}u^{(\beta^2-1)/\beta}\gamma dx = -\sigma\int_{\R^2} u\nabla p\cdot\nabla u^{\beta-1} dx\\
&= \frac{\sigma(\beta-1)}{\beta}\int_{\R^2}u^\beta\Delta p dx.
\end{align*}
On the other hand, multiplying \eqref{mar15.fp.2} by $\sigma\Delta p\in L^2(\R^2)$ and integrating in $\R^2$ yields
\begin{align*}
\sigma\int_{\R^2}u^\beta\Delta p dx &= \sigma\int_{\R^2}(\tau^{-1}(p - p^*) + (-\Delta)^{s}p - \varrho\Delta p )\Delta p dx \\
&= -\tau^{-1} \sigma \int_{\R^2} (\nabla p - \nabla p^*)\cdot\nabla p dx -\sigma \int_{\R^2}|(-\Delta)^{s/2}\nabla p|^2 dx - \sigma\black{\varrho_2}\int_{\R^2}(\Delta p)^2 dx.
\end{align*}
Given that 
$$ (u - u^*)u^{\beta-1} \geq u^\beta/\beta - (u^*)^\beta/\beta,\qquad  (\nabla p - \nabla p^*)\cdot\nabla p \geq |\nabla p|^2/2 - |\nabla p^*|^2/2 , $$
we deduce
\begin{align}
\frac{1}{\tau}\int_{\R^2}\left(\frac{u^\beta}{\beta}  + \sigma\frac{\beta-1}{2\beta}|\nabla p|^2\right) dx + \black{\varrho_1} \int_{\R^2}|\nabla u^{\beta-1}|^{(\beta+1)/\beta} dx 
+ \eps\int_{\R^2}u^{(\beta^2-1)/\beta}\gamma dx& \nonumber \\
+\frac{\sigma(\beta-1)}{\beta}\int_{\R^2}|(-\Delta)^{s/2}\nabla p|^2 dx  + \black{\varrho_2}\frac{\sigma(\beta-1)}{\beta}\int_{\R^2}(\Delta p)^2 dx& \label{mid_est}\\
\leq \; \frac{1}{\tau}\int_{\R^2}\left(\frac{(u^*)^\beta}{\beta} + \sigma\frac{\beta-1}{2\beta}|\nabla p^*|^2\right) dx .&\nonumber 
\end{align}
The above estimate yields a bound for $w=u^{\beta-1}$ in $Y$ which is uniform in $\sigma$. Together with the embedding $Y\hookrightarrow X$
we have that $u$ belongs to $X$, with $\|u\|_X$ bounded uniformly with respect to $\sigma$.
Leray-Schauder fixed point theorem yields the existence of a fixed point $w=u^{\beta-1}\in Y$ for $F(\cdot,1)$, i.e. a solution $(u,p)\in \tilde Y_\beta\times H^2(\R^2)$ to
\begin{align}\label{mar15.epstau.1}
\int_{\R^{2}}{\frac{u - u^*}{\tau}}\phi dx+\int_{\R^{2}} u\nabla p\cdot\nabla\phi dx &+ \black{\varrho_1} \int_{\R^2}|\nabla u^{\beta-1}|^{1/\beta-1}\nabla u^{\beta-1}\cdot\nabla\phi dx\\
\nonumber 
&+ \eps\int_{\R^2}u^{(\beta-1)/\beta}\phi\gamma dx = 0
\quad {\forall\phi\in Y} ,\\
\label{mar15.epstau.2}
 \frac{p - p^*}{\tau} + (-\Delta)^{s}p -& \black{\varrho_2}\Delta p - u^\beta = 0 \qquad\mbox{in }\R^2 ,
\end{align}
such that $u, p\geq 0$ a.e. in $\R^2$ and (\ref{mid_est}) holds for $\sigma=1$:
\begin{align}\label{mar15.ei.1}
\frac{1}{\tau}\int_{\R^2}\left(\frac{u^\beta}{\beta} + \frac{\beta-1}{2\beta}|\nabla p|^2\right) dx + \black{\varrho_1} \int_{\R^2}|\nabla u^{\beta-1}|^{\frac{\beta+1}{\beta}} dx + \eps\int_{\R^2}u^{(\beta^2-1)/\beta}\gamma dx& \\
\nonumber
\qquad +\frac{\beta-1}{\beta}\int_{\R^2}|(-\Delta)^{s/2}\nabla p|^2 dx + \frac{\black{\varrho_2(\beta-1)}}{\beta}\int_{\R^2}(\Delta p)^2 dx
\leq \frac{1}{\tau}\int_{\R^2}\left(\frac{(u^*)^\beta}{\beta} + \frac{\beta-1}{2\beta}|\nabla p^*|^2\right) dx &.
\end{align}

\subsection{The limit $\eps\to 0$} The next step is to take the $\lim_{\eps\to 0}$ in \eqref{mar15.epstau.1}-\eqref{mar15.ei.1}. 

The  uniform bound of $u^{\beta-1}$  in $W^{1,(1+\beta)/\beta}(\R^2)$  (see (\ref{mar15.ei.1})) and Sobolev's embedding insure that for every $R>0$ there exists a subsequence 
$u^{(\eps,R)}$ of $u^{(\eps)}$ such that
$$
u^{(\eps,R)}\to u\quad\mbox{strongly in }L^{q}(B_R),~~ 
1\leq q < 2(\beta+1) .~~
R>0,
$$
The function $u$ is the weak limit of $u^{(\eps)}$ in $L^{2(\beta+1)}(\R^2)$.
By a Cantor diagonal argument we can find a subsequence (not relabeled) of $u^{(\eps)}$ such that 
$$
u^{(\eps)}\to u\quad\mbox{strongly in }L^{q}(B_R),~~
1\leq q < 2(\beta+1),~~R\in\N,
$$
as well as $u^{(\eps)} \to u$ a.e.~in $\R^2$. As a consequence
\begin{align}
{(u^{(\eps)})^\beta} & {\to u^\beta \;\; \textrm{strongly in } \; L^2(B_R)},\quad u^{(\eps)} \to u \;\; \textrm{strongly in } \;  L^{2/s}(B_R),\quad R>0. \label{strongL2}
\end{align}
Going back to the limit in \eqref{mar15.epstau.2} and \eqref{mar15.epstau.1} we have that as $\eps \to 0$
\begin{align*}
 \int_{\R^2} (u^{(\eps)})^\beta \psi \;dx &\to \int_{\R^2} u^\beta \psi \;dx, \quad \forall \psi\in C^\infty_c(\R^2), \\
 \int_{\R^{2}} u^{(\eps)} \nabla p^{(\eps)} \cdot\nabla\phi dx& \to \int_{\R^2}\int_{\R^{2}} u\nabla p \cdot\nabla\phi dx, \quad \forall {\phi\in C^\infty_c(\R^2)},
\end{align*}
where we used (\ref{strongL2}) for the first limit, and (\ref{strongL2}) together with $\nabla p^{(\eps)} \rightharpoonup \nabla p$ in $L^{2/(1-s)}(\R^2)$ 
 to obtain the second limit (remember that $p^{(\eps)}$ is relatively weakly compact in $H^{1+s}(\R^2)$). Summarizing, taking the limit $\eps\to 0$ in \eqref{mar15.epstau.1}, \eqref{mar15.epstau.2}
and subsequently employing a standard density argument we get
\begin{align}\label{mar15.tau.1}
\int_{\R^{2}}( \tau^{-1}(u - u^*)\phi + u\nabla p\cdot\nabla\phi 
+ {\varrho_1}|\nabla u^{\beta-1}|^{1/\beta-1}
\nabla u^{\beta-1}\cdot\nabla\phi ) dx &= 0\quad 
\forall \phi\in W^{1,\frac{1+\beta}{\beta}}(\R^2) ,\\
\label{mar15.tau.2}
 \tau^{-1}(p - p^*) + (-\Delta)^{s}p - {\varrho_2}\Delta p - u^\beta &= 0 \qquad\mbox{in }\R^2 .
\end{align}
Moreover $u, p\geq 0$ a.e. in $\R^2$ and 
\begin{align}\label{mar15.ei.2}
\frac{1}{\tau}\int_{\R^2}\left(\frac{u^\beta}{\beta} + \frac{\beta-1}{2\beta}|\nabla p|^2\right) dx 
& + {\varrho_1} \int_{\R^2}|\nabla u^{\beta-1}|^{\frac{1+\beta}{\beta} } dx \\
\nonumber
\qquad +\frac{\beta-1}{\beta}\int_{\R^2}|(-\Delta)^{s/2}\nabla p|^2 dx + \frac{\black{\varrho_2}(\beta-1)}{\beta}\int_{\R^2}(\Delta p)^2 dx
\leq& \frac{1}{\tau}\int_{\R^2}\left(\frac{(u^*)^\beta}{\beta} + \frac{\beta-1}{2\beta}|\nabla p^*|^2\right) dx .
\end{align}
Let $G_\delta(x)\equiv\min\{x/\delta,1\}$ for every $x\geq 0$.
By testing \eqref{mar15.tau.2} against $G(p)\in L^2(\R^2)$ and exploiting the fact that $\int_{\R^2}G_\delta(p)(-\Delta)^s p\, dx\geq 0$ one deduces the estimate
$$
\int_{\R^2}G_\delta(p)p dx\leq \int_{\R^2}G_\delta(p)p^* dx
+\tau\int_{\R^2}G_\delta(p)u^\beta dx\leq
\int_{\R^2}p^* dx + \tau\int_{\R^2}u^\beta dx \leq C .
$$
Taking the limit $\delta\to 0$ in the above inequality (by monotone convergence) yields $p\in L^1(\R^2)$.

Let $\eta_R$ as in the statement of Lemma \ref{frac_lapl_eta_R}.
Multiplying \eqref{mar15.tau.2} by $\eta_R$, integrating in $\R^2$ and integrating by parts leads to
\begin{align}\label{mass_cons_tau}
 \tau^{-1}\int_{\R^2}(p - p^*)\eta_R dx = \int_{\R^2}(u^\beta\eta_R 
 + {\varrho_2} p \Delta\eta_R - p (-\Delta)^s\eta_R )dx. 
 \end{align}
Since $\|(-\Delta)^s\eta_R\|_{L^\infty}\to 0$ as $R\to\infty$ (see Lemma \ref{frac_lapl_eta_R}) and $p\in L^1(\R^2)$, the bound for the mass of $p$ follows
$$ \int_{\R^2}p dx = \int_{\R^2}p^* dx + \tau\int_{\R^2}u^\beta dx . $$
At this point we have proved the existence of sequences {$(u_k)_{k\in\N}\subset H^1(\R^2)$, $(p_k)_{k\in\N} \subset H^2(\R^2)$} such that $u_0=u_{in}$, $p_{0}=p_{in}$, and for $k\geq 1$
 $u_k, p_k\geq 0$ a.e. in $\R^2$,  
\begin{align}\label{mar15.d.1}
\int_{\R^{2}}( \tau^{-1}(u_k - u_{k-1})\phi + u_k\nabla p_k\cdot\nabla\phi)dx + {\varrho_1} \int_{\R^2}|\nabla u_k^{\beta-1}|^{1/\beta-1}\nabla u_k^{\beta-1}\cdot\nabla\phi dx&  = 0\quad \forall \phi\in H^1(\R^2) ,\\
\label{mar15.d.2}
 \tau^{-1}(p_k - p_{k-1}) + (-\Delta)^{s}p_k - {\varrho_2}\Delta p_k - u_k^\beta & = 0 \qquad\mbox{in }\R^2 ,
\end{align}
with the estimates 
\begin{align}
\label{mar15.dei.0}
\frac{1}{\tau}\int_{\R^2}\left(\frac{u_k^\beta}{\beta} + \frac{\beta-1}{2\beta}|\nabla p_k|^2\right) dx &+ {\varrho_1} \int_{\R^2}|\nabla u_k^{\beta-1}|^{(1+\beta)/\beta} dx \\
\nonumber
\quad +\frac{\beta-1}{\beta}\int_{\R^2}|(-\Delta)^{s/2}\nabla p_k|^2 dx + \frac{{\varrho_2}(\beta-1)}{\beta}\int_{\R^2}(\Delta p_k)^2 dx
& \leq \frac{1}{\tau}\int_{\R^2}\left(\frac{(u_{k-1})^\beta}{\beta} + \frac{\beta-1}{2\beta}|\nabla p_{k-1}|^2\right) dx ,\\
\int_{\R^2}p_k dx& =\int_{\R^2}p_{k-1} dx + \tau\int_{\R^2}u_k^\beta dx.\label{mar15.massk}
\end{align}
Choose $T>0$ arbitrary. Define $N=T/\tau$, $u^{(\tau)}(t) = u_0\chi_{\{0\}}(t) + \sum_{k=1}^N u_{k}\chi_{((k-1)\tau, k\tau]}(t)$,
$p^{(\tau)}(t) = p_0\chi_{\{0\}}(t) + \sum_{k=1}^N p_{k}\chi_{((k-1)\tau, k\tau]}(t)$.
Moreover define the backward finite difference w.r.t. time $D_\tau$ as
$$ D_\tau f (t) \equiv \tau^{-1}(f(t)-f(t-\tau)),\qquad t\in [\tau,T]. $$
We can rewrite \eqref{mar15.d.1}--\eqref{mar15.massk} with the new notation. For all $\phi\in L^2(0,T; H^1(\R^2))\cap L^{\frac{1+\beta}{\beta}}(0,T; W^{\frac{1+\beta}{\beta}}(\R^2))$ and  $ \psi\in L^2(0,T; H^1(\R^2))$ we have
\begin{align}\label{mar15.1}
&\int_0^T\int_{\R^{2}}( (D_\tau\utau)\phi + \utau\nabla \ptau\cdot\nabla\phi)dx dt \\
&\nonumber
\quad + {\varrho_1}\int_0^T \int_{\R^2}|\nabla (\utau)^{\beta-1}|^{1/\beta-1}\nabla (\utau)^{\beta-1}\cdot\nabla\phi \, dx dt = 0,\\
\label{mar15.2}
&\int_0^T\int_{\R^{2}}( (D_\tau\ptau)\psi + ((-\Delta)^{s/2}\ptau)((-\Delta)^{s/2}\psi) + {\varrho_2}\nabla\ptau\cdot\nabla\psi - (\utau)^\beta\psi)dx = 0, \\
\label{mar15.dei}
&\int_{\R^2}\left(\frac{(\utau)^\beta}{\beta} + \frac{\beta-1}{2\beta}|\nabla \ptau|^2\right) dx + {\varrho_1} \int_0^t\int_{\R^2}|\nabla (\utau)^{\beta-1}|^{(1+\beta)/\beta} 
dx dt' \\
\nonumber
&\qquad + \frac{{\varrho_2}(\beta-1)}{\beta}\int_0^t\int_{\R^2}(\Delta \ptau)^2 dx dt'
+\frac{\beta-1}{\beta}\int_0^t\int_{\R^2}|(-\Delta)^{s/2}\nabla \ptau|^2 dx dt' \\
\nonumber
&\qquad\leq \int_{\R^2}\left(\frac{(u_{in})^\beta}{\beta} + \frac{\beta-1}{2\beta}|\nabla p_{in}|^2\right) dx ,\\
& \int_{\R^2}\ptau(t) dx \,\le  \int_{\R^2}p_{in} dx + C t\qquad t\in [0,T],\label{mar15.mass.tau}
\end{align}
where the constant in \eqref{mar15.mass.tau} only depends on the entropy at initial time.\medskip\\

\subsection{The limit $\tau \to 0$} 
We first estimate the time derivative of the density function. Let $R>0$ arbitrary,
$Q_{R,T}\equiv B_R\times (0,T)$. For any $\phi\in C^\infty_c(Q_{R,T})$
\begin{align*}
&\left| \int_0^T\int_{\R^{2}} (D_\tau\utau)\phi  \;dxdt \right| \\ 
&\leq \left|  \int_0^T\int_{\R^{2}} \utau\nabla \ptau\cdot \nabla \phi \;dxdt \right|
 + {\varrho_1}  \left|  \int_0^T\int_{\R^{2}} |\nabla (\utau)^{\beta-1}|^{1/\beta-1} \nabla (\utau)^{\beta-1} \cdot \nabla \phi \;dxdt \right|  \\
&\le \|\nabla \ptau \|_{L^2(0,T; L^{\frac{2}{1-s}}(\R^2))}
\| \utau \|_{L^{\infty}(0,T; L^{\beta+1}(\R^2))}
\| \nabla \phi \|_{L^{2}(0,T; L^{\frac{2\beta}{2+(1-s)\beta}}(\R^2))}\\
&\qquad + \varrho_1 \|\nabla (\utau)^{\beta-1} \|_{L^{\frac{\beta+1}{\beta}}(\R^2)}^{\frac{1}{\beta}} 
\|\na \phi \|_{L^{\frac{\beta+1}{\beta}}(\R^2)} \\
&\le C(T) \|\phi\|_{L^{2}(0,T;\, W^{1,\frac{\beta+1}{\beta}}\cap W^{1,\frac{2\beta}{2+(1-s)\beta}}(\R^2))}
\end{align*}
using \eqref{mar15.dei}. This yields 
\begin{align}
\label{est.ut}
\| D_\tau\utau \|_{L^{2}(0,T;\, (W^{1,\frac{\beta+1}{\beta}}\cap W^{1,\frac{2\beta}{2+(1-s)\beta}}(\R^2))' ) } &\le C(T).
\end{align}
In particular 
\begin{align}
\label{est.ut.loc}
\| D_\tau\utau \|_{L^{2}(0,T; 
W^{-1,\frac{\lambda}{\lambda-1} }(B_R))} \le C(T,R),\quad
\forall R>0,\quad\lambda\equiv\max\left\{ \frac{\beta+1}{\beta} , \frac{2\beta}{2+(1-s)\beta}\right\} . 
\end{align}
The compact Sobolev embedding $W^{1,2(\beta+1)/\beta}(B_R)\hookrightarrow L^{2(\beta+1)/\beta-\epsilon}(B_R)$, valid for every $\epsilon>0$, allows us to apply Aubin-Lions Lemma in the version of \cite{CheJueLiu14}
and obtain, for any $R>0$, the existence of a subsequence $u^{(\tau,R)}$ of $\utau$ such that 
$$ u^{(\tau,R)} \to u \quad \textrm{strongly in }\; L^2(0,T; L^{2}(B_R)). $$
The limit function $u$ is unique and coincides with the weak-* limit of $u^{(\tau)}$
in $L^\infty(0,T; L^\beta(\R^2))$. A Cantor diagonal argument allows us to find a subsequence of $\utau$ (which we denote again with $\utau$) such that 
$$ u^{(\tau)} \to u \quad \textrm{strongly in }\; L^1(0,T; L^{1}(B_R)),\qquad
\forall R\in\N , $$
and
\begin{align}\label{aetau}
\utau \to u \quad \textrm{a.e. in } \; \R^2 \times [0,T].
\end{align}
Since $\utau \in L^\infty(0,T,L^{\beta}(\R^2)) \cap L^{\frac{\beta^2-1}{\beta}}(0,T,L^{2(\beta+1)}(\R^2))$, a straightforward interpolation yields 
\begin{align}\label{int_u^2}
\|\utau\|_{L^r(0,T,L^{r}(\R^2))} \le C,\qquad
r = \frac{3\beta^2 + \beta - 2}{2\beta}.
\end{align}
Since $r>\beta$, thanks to (\ref{aetau}) it follows
\begin{align}\label{vac.3}
\utau\to u \quad \textrm{strongly in}\; L^\beta(0,T; L^{\beta}(B_R)),\quad\forall R>0.
\end{align}
Hence as $\tau \to 0$:
$$
\int_0^T \int_{\R^2} (\utau)^\beta \psi \;dxdt  \to \int_0^T \int_{\R^2} u^\beta \psi \;dxdt,\quad \textrm{for all}\; \psi \in C^0_c(\R^2\times (0,T)).
$$
Moreover directly from (\ref{mar15.dei}) 
\begin{align*}
(\utau)^{\beta-1} &\rightharpoonup u^{\beta-1}\quad 
\mbox{weakly in }L^{(\beta+1)/\beta}(0,T; W^{1,(\beta+1)/\beta}(\R^2)),\\
\utau &\rightharpoonup u\quad \mbox{weakly* in }L^\infty(0,T; L^\beta(\R^2)).
\end{align*}
From  \eqref{est.ut}, \eqref{vac.3} it follows
$$ D_\tau \utau  \rightharpoonup \partial_t u  \quad \textrm{weakly in }\; 
L^{2}(0,T;\, (W^{1,\frac{\beta+1}{\beta}}\cap W^{1,\frac{2\beta}{2+(1-s)\beta}}(\R^2))' ). $$
%
 Since $\ptau$ is uniformly bounded in $L^\infty(0,T,L^1(\R^2))$ and $\nabla \ptau$ is uniformly bounded in $L^\infty(0,T,L^2(\R^2))$, 
 {Gagliardo-Nirenberg} and the entropy inequality \eqref{mar15.dei} yield 
 \begin{align}\label{L2_ptau}
 \| \ptau \|_{L^\infty(0,T,H^1(\R^2))} + \|\ptau\|_{L^2(0,T; H^{s+1}(\R^2))}
 + \sqrt{\varrho_2}\|\ptau\|_{L^2(0,T; H^{2}(\R^2))} \le C, 
 \end{align}
 where $C$ only depends on the initial data. Hence there exists a subsequence of $\ptau$ (which we denote again with $\ptau$) such that
\begin{align*}
\ptau &\rightharpoonup p\quad\mbox{weakly in }L^2(0,T; H^{s+1}(\R^2) ),\\
\ptau &\rightharpoonup^* p\quad\mbox{weakly* in }L^{\infty}(0,T,H^1(\R^2)) . 
\end{align*}
In particular
\begin{align}
%
\label{L2_p} 
\| p \|_{L^\infty(0,T,H^1(\R^2))} + \| p \|_{L^2(0,T,H^{s+1}(\R^2))}  \le C.
\end{align}
Also, by Sobolev's embedding,
$$
\na\ptau\rightharpoonup \na p\quad\mbox{weakly in }L^2(0,T; L^{2/(1-s)}(\R^2)) .
$$ 
The strong convergence $\utau\to u$ in $L^{2}(0,T; L^\beta(B_R))$ for every $R>0$,
the weak convergence of $\na\ptau$ in $L^2(0,T; L^{2/(1-s)}(\R^2))$, and the assumption $s > \frac{1}{\beta}$ imply
$$
\int_0^T\int_{\R^{2}} \utau\nabla \ptau \cdot\nabla\phi\;dx dt
\to
\int_0^T\int_{\R^{2}} u \nabla p\cdot\nabla\phi\;dx dt,
\quad \textrm{for all}\; \phi \in C^1_c(\R^2\times (0,T)).
$$
Let us look at the discrete time derivatives of the pressure function. Thanks to (\ref{int_u^2}) we have 
  \begin{align*}
\left| \int_0^T\int_{\R^2}(\utau)^\beta\psi dx dt \right|
 &\leq \|(\utau)^\beta\|_{L^{r/\beta}(\R^2\times (0,T))} 
 \|\psi\|_{L^{r/(r-\beta)}(\R^2\times (0,T))}\\
 &\leq C(\varrho_1)\|\psi\|_{L^{r/(r-\beta)}(\R^2\times (0,T))},
 \end{align*}
while \eqref{mar15.dei} implies
\begin{align*}
\left|\int_0^T\int_{\R^2}(-\Delta)^s\ptau\,\psi dx dt \right| + 
\varrho_2\left|\int_0^T\int_{\R^2}\Delta\ptau\,\psi dx dt\right| \leq
C\|\psi\|_{L^2(\R^2\times (0,T))} .
\end{align*}
We deduce
\begin{align}\label{boh_p}
\left| \int_0^T\int_{\R^{2}} (D_\tau \ptau)\psi  \;dxdt \right|  \le 
C(\varrho_1) \|\psi\|_{L^{2}\cap L^{r/(r-\beta)}(\R^2\times (0,T))}.
\end{align}
It follows
\begin{align}
\label{}
D_\tau\ptau\rightharpoonup \pa_t p\quad\mbox{weakly in }(L^{2}\cap L^{r/(r-\beta)}(\R^2\times (0,T)))'.
\end{align}
Since $\ptau$ is bounded in $L^\infty(0,T,H^1(\R^2))$ and $D_\tau \ptau$ is bounded in $(L^2\cap L^{r/(r-\beta)}(\R^2\times (0,T)))'$, 
we can invoke Aubin-Lions lemma to deduce, for every $R\in\N$, the existence of a subsequence $p^{(\tau,R)}$ of $\ptau$ such that $p^{(\tau,R)} \to p$ strongly in $L^1(0,T,L^1(B_R))$, for every $R\in\N$. A Cantor's diagonal argument yields 
the existence of a subsequence of $\ptau$ (which we call again $\ptau$) such that
\begin{align}\label{p_a.e.}
\ptau \to p\; \textrm{strongly in}\; L^1(0,T,L^1(B_R))\quad\forall R\in\N ,\quad \ptau \to p \; \textrm{a.e in } \; \R^2.
\end{align}
%
%
At this point we can take the limit $\tau\to 0$ in \eqref{mar15.1} and \eqref{mar15.2}, which yields (after a suitable density argument)
\begin{align}\label{mar15.eq.1}
\int_0^T & \langle\pa_t u,\phi\rangle dt
+\int_0^T\int_{\R^{2}}u\nabla p\cdot\nabla\phi\, dx dt + {\varrho_1} \int_0^T \int_{\R^2}|\nabla u^{\beta-1}|^{1/\beta-1}\nabla u^{\beta-1}\cdot\nabla\phi \, dx dt = 0\\
\nonumber
&\qquad\forall\phi \in L^2(0,T; W^{1,\frac{2\beta}{2+(1-s)\beta}}(\R^2))\cap L^{\frac{1+\beta}{\beta}}(0,T; W^{\frac{1+\beta}{\beta}}(\R^2)) ,\nonumber\\ 
\label{mar15.eq.2}
\int_0^T &\langle\pa_t p,\psi\rangle dt + \int_0^T\int_{\R^2}( (-\Delta)^{s} p - u^\beta)\psi dx dt
- \varrho_2\int_0^T\int_{\R^2}\psi\Delta p\, dx dt = 0 \\
&\qquad\forall\psi \in L^{2}\cap L^{\frac{r}{r-\beta}}(\R^2\times (0,T)),\nonumber
\end{align}
where $r = \frac{3\beta^2 + \beta - 2}{2\beta}$ is defined in \eqref{int_u^2}.

Thanks to the lower weak semicontinuity of the $L^p$ norm we deduce from \eqref{mar15.dei} the following entropy inequality:
\begin{align}\label{mar15.ei}
&\int_{\R^2}\left(\frac{u^\beta}{\beta} + \frac{\beta-1}{2\beta}|\nabla p|^2\right) dx + \black{\varrho_1} \int_0^t\int_{\R^2}|\nabla u|^{\frac{\beta+1}{\beta}} dx dt' 
+ \frac{\black{\varrho_2}(\beta-1)}{\beta}\int_0^t\int_{\R^2}(\Delta p)^2 dx dt'\\
\nonumber
&\qquad +\frac{\beta-1}{\beta}\int_0^t\int_{\R^2}|(-\Delta)^{s/2}\nabla p|^2 dx dt' 
\leq \int_{\R^2}\left(\frac{(u_{in})^\beta}{\beta} + \frac{\beta-1}{2\beta}|\nabla p_{in}|^2\right) dx .
\end{align}
{{Furthermore, thanks to the a.e. convergence of $\ptau$ (\ref{p_a.e.}) we can apply Fatou's Lemma in \eqref{mar15.mass.tau} and get
\begin{align}
\int_{\R^2}p(t) dx  \,{\le}  \int_{\R^2}p_{in} dx &+ C t,\qquad t\in [0,T].\label{mar15.mass}
\end{align}}}

\subsection{The limit $\black{\varrho_2} \to 0$} 

From the entropy inequality \eqref{mar15.ei} and the mass conservation \eqref{mar15.mass} we deduce the following $\rho_2-$uniform bounds:
\begin{align}
\label{est.apr11.u}
\|u\|_{L^\infty(0,T; L^2(\R^2))} + 
\|u^{\beta-1}\|_{L^{\frac{\beta+1}{\beta}}(0,T; W^{1,\frac{\beta+1}{\beta}}(\R^2))} &\leq C(\rho_1,T) ,\\
\label{est.apr11.p}
\|p\|_{L^\infty(0,T; H^1(\R^2))} + 
\|p\|_{L^2(0,T; H^{s+1}(\R^2))} + 
\sqrt{\rho_2}\|\Delta p\|_{L^2(0,T; L^2(\R^2))} &\leq C(T).
\end{align}
Moreover, from \eqref{mar15.eq.1}, \eqref{mar15.eq.2}, \eqref{est.apr11.u}, \eqref{est.apr11.p} we deduce $\rho_2-$uniform bounds
for the time derivatives of $u$, $p$:
\begin{align}
\|\pa_t u\|_{L^{2}(0,T;\, (W^{1,\frac{\beta+1}{\beta}}\cap W^{1,\frac{2\beta}{2+(1-s)\beta}}(\R^2))' )} + 
\|\pa_t p\|_{L^2(0,T;\, (L^2 \cap L^{\frac{r}{r-\beta}}(\R^2))')} \leq C(T,\rho_1).
\label{est.apr11.dt}
\end{align}
Estimates \eqref{est.apr11.u}--\eqref{est.apr11.dt}
and the compact Sobolev embeddings $W^{1,(\beta+1)/\beta}(\Omega)\hookrightarrow L^{2(\beta+1)/(\beta-1)-\epsilon}(\Omega)$, $H^{s+1}(\Omega)\hookrightarrow W^{1,2/(1-s)-\epsilon}(\Omega)$,
valid for every bounded open $\Omega\subset\R^2$ and $\epsilon>0$,
allow us to apply Aubin-Lions Lemma and deduce, for every $R\in\N$, the existence of subsequences $u^{(\rho_2,R)}$, $p^{(\rho_2,R)}$ of $u^{(\rho_2)}$, $p^{(\rho_2)}$ such that
\begin{align*}
u^{(\rho_2,R)}\to u\quad\mbox{strongly in }L^1(0,T; L^{1}(B_R)),
\quad 
p^{(\rho_2,R)}\to p\quad\mbox{strongly in }L^1(0,T; L^1(B_R)),
\end{align*}
for every $R\in\N$. Once again, a Cantor diagonal argument allows us to find subsequences (not relabeled) of $u^{(\rho_2)}$, $p^{(\rho_2)}$ such that
\begin{align*}
u^{(\rho_2)}\to u\quad\mbox{strongly in }L^1(0,T; L^{1}(B_R)),
\quad 
p^{(\rho_2)}\to p\quad\mbox{strongly in }L^1(0,T; L^{1}(B_R)),
\end{align*}
for every $R\in\N$.
Bounds \eqref{est.apr11.u}, \eqref{est.apr11.p} also imply (up to subsequences) the following weak convergence relations 
\begin{align*}
u^{(\rho_2)} &\rightharpoonup^* u\quad\mbox{weakly-* in }L^\infty(0,T; L^{2}(\R^2)),\\
\na (u^{(\rho_2)})^{\beta-1} &\rightharpoonup \na u^{\beta-1}\quad\mbox{weakly in }L^{(1+\beta)/\beta}(0,T; L^{(1+\beta)/\beta}(\R^2)),\\
p^{(\rho_2)} &\rightharpoonup^* p\quad\mbox{weakly-* in }L^\infty(0,T; H^{1}(\R^2)),\\
p^{(\rho_2)} &\rightharpoonup p\quad\mbox{weakly in }L^2(0,T; H^{s+1}(\R^2)).
\end{align*}
Thanks to the convergence relations stated above,
taking the limit $\rho_2\to 0$ in \eqref{mar15.eq.1}, \eqref{mar15.eq.2} is at this point straightforward and leads to
\begin{align}\label{apr11.eq.1}
	\int_0^T & \langle\pa_t u,\phi\rangle dt
	+\int_0^T\int_{\R^{2}}u\nabla p\cdot\nabla\phi\, dx dt + {\varrho_1} \int_0^T \int_{\R^2}|\nabla u^{\beta-1}|^{1/\beta-1}\nabla u^{\beta-1}\cdot\nabla\phi \, dx dt = 0\\
	\nonumber
	&\qquad\forall\phi \in L^2(0,T; W^{1,\frac{2\beta}{2+(1-s)\beta}}(\R^2))\cap L^{\frac{1+\beta}{\beta}}(0,T; W^{\frac{1+\beta}{\beta}}(\R^2)),\nonumber\\ 
	\label{apr11.eq.2}
	\int_0^T &\langle\pa_t p,\psi\rangle dt + \int_0^T\int_{\R^2}( (-\Delta)^{s} p - u^\beta)\psi dx dt = 0 \\
	&\qquad\forall\psi \in L^{2}\cap L^{\frac{r}{r-\beta}}(\R^2\times (0,T)),\nonumber
\end{align}
where $r = \frac{3\beta^2 + \beta - 2}{2\beta}$ is defined in \eqref{int_u^2}.

The same convergence relations yield
\begin{align}\label{apr11.ei}
&\int_{\R^2}\left(\frac{u^\beta}{\beta} + \frac{\beta-1}{2\beta}|\nabla p|^2\right) dx + \black{\varrho_1} \int_0^t\int_{\R^2}|\nabla u^{\beta-1}|^{\frac{\beta+1}{\beta}} dx dt' \\
\nonumber
&\qquad +\frac{\beta-1}{\beta}\int_0^t\int_{\R^2}|(-\Delta)^{s/2}\nabla p|^2 dx dt' 
\leq \int_{\R^2}\left(\frac{(u_{in})^\beta}{\beta} + \frac{\beta-1}{2\beta}|\nabla p_{in}|^2\right) dx .
\end{align}
We also point out that \eqref{mar15.mass} holds true also after taking the limit $\rho_2\to 0$.

\subsection{The limit $\black{\varrho_1} \to 0$} 

In the rest of the paper we denote $\rho_1$ with $\rho$. 

As a preliminary step, we are going to prove a uniform bound for $\na\prho$.
By interpolation we obtain
\begin{align*}
\|\na\prho\|_{L^{q}(\QT)}\leq 
\|\na\prho\|_{L^\infty(0,T; L^2(\R^2))}^\lambda
\|\na\prho\|_{L^{(1-\lambda)q}(0,T; L^{\frac{2}{1-s}}(\R^2))}^{1-\lambda},
\end{align*}
with $\frac{1}{q} = \frac{\lambda}{2}+\frac{(1-\lambda)(1-s)}{2}$, 
$0\leq\lambda\leq 1$. 
The assumption $s > \beta^{-1}$ allows for the choice  
$q>2(\beta+1)/\beta$ such that $(1-\lambda)q\leq 2$
and therefore
\begin{align*}
\|\na\prho\|_{L^{2(\beta+1)/\beta+\epsilon}(\QT)}\leq C
\|\na\prho\|_{L^\infty(0,T; L^2(\R^2))}^\lambda
\|\na\prho\|_{L^{2}(0,T; L^{2/(1-s)}(\R^2))}^{1-\lambda}\quad
\forall\epsilon\in [0,\epsilon_0),
\end{align*}
for some $\epsilon_0>0$.
Since $\na\prho$ is bounded in $L^2(0,T; H^s(\R^2))$, by Sobolev's embedding
it is also bounded in $L^2(0,T; L^{2/(1-s)}(\R^2))$. Together with the uniform bound
in $L^\infty(0,T; L^2(\R^2))$, we conclude
\begin{align}
\label{vac.aa}
\exists\epsilon_0>0\, : \quad
\|\na\prho\|_{L^{2(\beta+1)/\beta + \epsilon}(\QT)}\leq C\quad
\forall\epsilon\in [0,\epsilon_0).
\end{align}
Now we wish to prove a uniform bound for $\urho$ in $L^{\beta+1}(\R^2\times (0,T))$. Let us choose $\phi=\prho$, $\psi=\urho$ in
\eqref{apr11.eq.1}, \eqref{apr11.eq.2}, respectively, and sum the resulting equations.
We obtain
\begin{align}
\label{vac.a}
\int_0^T\int_{\R^2} & (\urho)^{\beta+1}dx dt = 
\int_0^T\int_{\R^2}\urho (-\Delta)^s\prho dx dt + 
\int_0^T\int_{\R^2}\urho |\na\prho|^2 dx dt\\
\nonumber
& +
\int_{\R^2}\urho(T)\prho(T)dx - \int_{\R^2}u_{in}p_{in}dx +
\rho\int_0^T\int_{\R^2}\na (\urho)^{\beta-1}\cdot\na\prho dx dt .
\end{align}
Let us bound the terms on the right-hand side of \eqref{vac.a} by using bounds
\eqref{mar15.mass}, \eqref{apr11.ei}. Applying H\uml older and Gagliardo-Nirenberg inequalities yields
\begin{align*}
\int_0^T\int_{\R^2} & \urho (-\Delta)^s\prho dx dt\\
&\leq
\|\urho\|_{L^\infty(0,T; L^\beta(\R^2))}\|(-\Delta)^s\prho\|_{L^1(0,T; L^{\frac{\beta}{\beta-1}}(\R^2))}\\
&\leq C
\|\urho\|_{L^\infty(0,T; L^\beta(\R^2))}
\|\prho\|_{L^2(0,T; H^{1+s}(\R^2))}^\alpha
\|\prho\|_{L^\infty(0,T; L^1(\R^2))}^{1-\alpha}
\\
&\leq C ,
\end{align*}
for some $\alpha\in [0,1]$. Let us then consider
\begin{align*}
\int_0^T\int_{\R^2} & \urho |\na\prho|^2 dx dt\leq
\|\urho\|_{L^{\beta+1}(\QT)}\|\na\prho\|_{L^{2(\beta+1)/\beta}(\QT)}^2\leq 
C\|\urho\|_{L^{\beta+1}(\QT)} 
\end{align*}
thanks to \eqref{vac.aa}. Next we notice that
\begin{align*}
\int_{\R^2}\urho(T)\prho(T)dx\leq \|\urho\|_{L^\infty(0,T; L^2(\R^2))}
\|\prho\|_{L^\infty(0,T; L^2(\R^2))}\leq C .
\end{align*}
Finally, Gagliardo-Nirenberg inequality allows us to write
\begin{align*}
\rho\int_0^T\int_{\R^2} & |\na (\urho)^{\beta-1}|^{1/\beta-1}\na (\urho)^{\beta-1}\cdot\na\prho dx dt\\
&\leq
\rho\|\na (\urho)^{\beta-1}\|_{L^{\frac{\beta+1}{\beta}}(\QT)}^{1/\beta}
\|\na\prho\|_{L^\frac{\beta+1}{\beta}(\QT)}\\
&\leq
C\rho\|\na (\urho)^{\beta-1}\|_{L^{\frac{\beta+1}{\beta}}(\QT)}^{1/\beta}
\|\prho\|_{L^\infty(0,T; L^1(\R^2))}^{\alpha}
\|\prho\|_{L^{2}(0,T; H^{1+s}(\R^2))}^{1-\alpha}\\
&\leq C,
\end{align*}
for some $\alpha\in [0,1]$.
From \eqref{vac.a} we conclude
\begin{align*}
\|\urho\|_{L^{\beta+1}(\QT)}^{\beta+1}\leq C_1\|\urho\|_{L^{\beta+1}(\QT)} + C_2
\end{align*}
which implies, via Young's inequality,
\begin{align}
\label{vac.b}
\|\urho\|_{L^{\beta+1}(\QT)}\leq C.
\end{align}
Next we find a suitable bound for $\urho\na\prho$. Since $\urho$ and $\na\prho$
are bounded in $L^\infty(0,T; L^\beta(\R^2))$ and $L^\infty(0,T; L^2(\R^2))$,
respectively, then $\urho\na\prho$ is bounded in $L^\infty(0,T; L^{2\beta/(2+\beta)}(\R^2))$.
On the other hand, $\urho$ and $\na\prho$
are also bounded in $L^{\beta+1}(0,T; L^{\beta+1}(\R^2))$ 
and $L^2(0,T; L^{2/(1-s)}(\R^2))$,
respectively, so $\urho\na\prho$ is also bounded in $L^{2(\beta+1)/(\beta+3)}(0,T; L^{2(\beta+1)/(2+(1-s)(\beta+1)}(\R^2))$. A straightforward interpolations leads to
\begin{align}
\label{vac.c}
\|\urho\na\prho\|_{L^{\frac{2(1+\beta)((1+s)\beta + 2)}{(\beta+2)(\beta+3)}}(\QT)}\leq C.
\end{align}
Now we prove the strong convergence of $\prho$.
From \eqref{apr11.ei}, \eqref{vac.b} it follows that
\begin{align}
	\label{vac.pat}
\|\pa_t\prho\|_{L^{(\beta+1)/\beta}(0,T; (L^2\cap L^{\beta+1}(\R^2))'}\leq C
\end{align}
From \eqref{apr11.ei} and \eqref{vac.pat} we deduce via Aubin-Lions Lemma and a Cantor diagonal argument that, up to subsequences,
$$
\prho\to p\quad\mbox{strongly in }L^1(B_R\times (0,T)),\quad\forall R>0.
$$
Bound \eqref{vac.b} implies that, up to subsequences, 
\begin{align}\label{lim.u}
\urho &\rightharpoonup u\quad\mbox{weakly in }L^{\beta+1}(\QT),\\
\nonumber
(\urho)^\beta &\rightharpoonup v\quad\mbox{weakly in }L^{\frac{\beta+1}{\beta}}(\QT),
\end{align}
for some function $v\in L^{\frac{\beta+1}{\beta}}(\QT)$. We are now going to show that 
$v=u^\beta$ a.e.~in $\QT$.

Let us now consider the vector fields
\begin{align*}
U^{(\rho)} &\equiv (\urho, -\urho\nabla\prho),\quad
V^{(\rho)} \equiv (\pa_t\prho, \na\prho) .
\end{align*}
Let $\pi_0 = \frac{1+\beta}{\beta}$, 
$\pi_1=\pi_2=\left(\frac{2(1+\beta)((1+s)\beta + 2)}{(\beta+2)(\beta+3)}\right)'
= \left( 1 - \frac{(\beta+2)(\beta+3)}{2(1+\beta)((1+s)\beta + 2)} \right)^{-1} $.
It is easy to see that $\pi_{min}\equiv\min\{\pi_0,\pi_1,\pi_2\} = \pi_0$.
Let $\Omega\subset\R^2$ bounded open smooth domain.

Bound \eqref{vac.c} means that $\urho\pa_{x_i}\prho$ is bounded in $L^{\pi_i'}(\OmT)$,
for $i=1,2$, while $\urho$ is bounded in $L^{\pi_0'}(\OmT)$ thanks to \eqref{vac.b}.
In particular $U_i^{(\rho)}$ is bounded in $L^{\pi_i'}(\OmT)$ for $i=0,1,2$.

On the other hand, \eqref{apr11.ei} and \eqref{vac.b} imply that $\pa_t\prho$
is bounded in $L^{\pi_0}(\OmT)$, while $\pa_{x_i}\prho$ is bounded in $L^{\pi_i}(\OmT)$ for $i=1,2$ thanks to \eqref{vac.aa} and the trivial relation
$\pi_1=\pi_2\leq \frac{2(\beta+1)}{\beta}$ which holds thanks to the hypothesis
$s\geq \frac{1}{\beta}$. It follows that $V^{(\rho)}_i$ is bounded in 
$L^{\pi_i}(\OmT)$ for $i=0,1,2$.

Next we notice that
$$
\Div_{(t,x)}U^{(\rho)} = \Div_x\left( \rho |\nabla u^{\beta-1}|^{1/\beta-1}\nabla u^{\beta-1} \right)\to 0\quad\mbox{strongly in }W^{-1,\pi_{min}'}(\OmT)
$$
thanks to \eqref{apr11.ei}. On the other hand $\Curl V^{(\rho)}\equiv 0$
since $V^{(\rho)}$ is a gradient field. 

Therefore we are able to apply \cite[Thr.~1.1]{GasMar08} and deduce that
$$
U^{(\rho)}\cdot V^{(\rho)}\rightharpoonup U\cdot V\quad\mbox{in }\mathcal{D}'(\OmT),
$$
where $U$, $V$ are the weak limits of $U^{(\rho)}$, $V^{(\rho)}$, respectively.
This implies, being $U^{(\rho)}\cdot V^{(\rho)}$ bounded in $L^1(\OmT)$,
\begin{align}\label{comm.1}
\urho\pa_t\prho - \urho |\na\prho|^2 \rightharpoonup 
u\pa_t p - \overline{\urho\na\prho}\cdot\na p \quad\mbox{in }\mathcal{D'}(\OmT),
\end{align}
where $u$, $\overline{\urho\na\prho}$ are the weak limits of $\urho$, $\urho\na\prho$, respectively. However,
we know that $\prho\to p$ strongly in $L^1(\OmT)$,
while $\na\prho$ is bounded in $L^{2(\beta+1)/\beta+\epsilon}(\R^2\times (0,T))$
and $L^2(0,T; H^s(\R^2))$ thanks to \eqref{apr11.ei}, \eqref{vac.aa}. Therefore Gagliardo-Nirenberg inequality allows us to deduce $\na\prho\to\na p$ strongly in 
$L^{2(\beta+1)/\beta}(\OmT)$. It follows
\begin{equation}\label{lim.unap}
\urho |\na\prho|^2\rightharpoonup u |\na p|^2,\quad
\urho\na\prho\rightharpoonup u\na p .
\end{equation}
From the relations above and \eqref{comm.1} we deduce
\begin{align}\label{comm.2}
	\urho\pa_t\prho \rightharpoonup 
	u\pa_t p \quad\mbox{in }\mathcal{D'}(\OmT).
\end{align}
Again, the local-in-space strong convergence of $\prho$ and the known uniform bounds for $\prho$ in $L^\infty(0,T; L^1(\R^2))$ and $L^2(0,T; H^{1+s}(\R^2))$ imply via Gagliardo-Nirenberg inequality that $(-\Delta)^s\prho\to (-\Delta)^s p$ strongly in $L^2(\OmT)$. This fact, together with the weak convergence $\urho \rightharpoonup u$ in $L^{\beta+1}(\QT)$ and relation $\beta\geq 2$, implies that
\begin{align}
	\label{comm.3}
	\urho (-\Delta)^s\prho\rightharpoonup u (-\Delta)^s p\quad\mbox{in }\mathcal{D'}(\OmT).
\end{align}
Summing \eqref{comm.2}, \eqref{comm.3}, employing \eqref{apr11.eq.2} and the uniform bound for $\urho$ in $L^{\beta+1}(\R^2\times (0,T))$ leads to
\begin{align}
	\label{comm.4}
	(\urho)^{\beta+1}\rightharpoonup u v\quad\mbox{in }\mathcal{M}(\OmT),
\end{align}
where $v$ is the weak limit of $(\urho)^\beta$ and $\mathcal{M}(\OmT)$ is the space of Radon measures, i.e.~the dual of $C^0_c(\OmT)$.

We are going to show that \eqref{comm.4} implies the a.e.~convergence of $\urho$ in $\QT$. Define the truncation operator $T_k$ as $T_k(x)\equiv\min\{x,k\}$ for every $x\geq 0$, $k\in\N$. 
Let $\phi\in C^0_c(\QT)$, $\phi\geq 0$ in $\QT$ arbitrary. 
Relation \eqref{comm.4} implies 
$$
\II{\urho T_k(\urho)^{\beta}\phi}\leq \II{(\urho)^{\beta+1}\phi}
\to\II{u v\phi}\quad\mbox{as }\rho\to 0,
$$
and so
\begin{equation}\label{bug.1}
\II{\overline{\urho T_k(\urho)^{\beta}}\phi}\leq \II{u v\phi} .
\end{equation}
On the other hand \cite[Thr.~10.19]{FeiNov} implies
\begin{equation}\label{bug.2}
\II{\overline{\urho T_k(\urho)^{\beta}}\phi}\geq 
\II{u\overline{T_k(\urho)^{\beta}}\phi}.
\end{equation}
The weak lower semicontinuity of the $L^1$ norm yields
\begin{align*}
\II{|\overline{T_k(\urho)^{\beta}} - v|} &\leq
\liminf_{\rho\to 0}\II{|T_k(\urho)^{\beta} - (\urho)^\beta|}\\
&\leq 2\liminf_{\rho\to 0}\int_{\{\urho>k\}}(\urho)^\beta dx dt\\
&\leq \frac{2}{k}\liminf_{\rho\to 0}\int_{\{\urho>k\}}(\urho)^{\beta+1} dx dt .
\end{align*}
The uniform bound for $\urho$ in $L^{\beta+1}(\R^2\times (0,T))$
implies
$$ \lim_{k\to\infty}\II{|\overline{T_k(\urho)^{\beta}} - v|} = 0, $$
which implies $\overline{T_k(\urho)^{\beta}}\rightharpoonup v$ weakly in $L^{\frac{\beta+1}{\beta}}(\QT)$ as $k\to\infty$, and so
\begin{align}\label{bug.3}
\II{u\overline{T_k(\urho)^{\beta}}\phi}\to \II{u v\phi}\quad\mbox{as }k\to\infty .
\end{align}
From \eqref{bug.1}--\eqref{bug.3} we deduce
\begin{equation*}
\lim_{k\to\infty}\II{(\overline{\urho T_k(\urho)^\beta}-u\overline{T_k(\urho)^\beta})\phi} = 0,
\end{equation*}
which easily implies
\begin{align}\label{bug.4}
\lim_{k\to\infty}
\lim_{\rho,\rho'\to 0}\II{(\urho - \urhop)(T_k(\urho)^\beta - T_k(\urhop)^\beta)\phi} = 0.
\end{align}
However, elementary computations yield
$$ 0\leq
(x-y)(T_k(x)^{\beta-1}-T_k(y)^{\beta-1})\leq 
(x-y)(T_{k+1}(x)^{\beta-1}-T_{k+1}(y)^{\beta-1})\quad\mbox{for }x, y\geq 0,~~ k\in\N ,
$$
which implies that the sequence
$$
a_k\equiv
\lim_{\rho,\rho'\to 0}\II{(\urho - \urhop)(T_k(\urho)^\beta - T_k(\urhop)^\beta)\phi}
$$
is nondecreasing and nonnegative. Moreover $\lim_{k\to\infty}a_k=0$ thanks to \eqref{bug.4}. Therefore $a_k=0$ for every $k\in\N$, that is
\begin{align*}
&\lim_{\rho,\rho'\to 0}\II{(\urho - \urhop)(T_k(\urho)^\beta - T_k(\urhop)^\beta)\phi} = 0, \qquad k\in\N.
\end{align*}
In particular
\begin{align}\label{bug.6}
&\lim_{\rho,\rho'\to 0}\iint_{\{M_{\rho,\rho'}\leq k\}}(\urho - \urhop)((\urho)^\beta - (\urhop)^\beta)\phi\, dx dt = 0, \qquad k\in\N,
\end{align}
where we defined $M_{\rho,\rho'}\equiv \max(\urho,\urhop)$.

It is easy to prove the elementary relation
$$
\frac{x^\beta - y^\beta}{x-y}\geq \max(x,y)^{\beta-1},\quad x,y\geq 0,~~ x\neq y ,
$$
which, together with \eqref{bug.6}, leads to
\begin{equation}\label{bug.7}
\lim_{\rho,\rho'\to 0}\iint_{\{M_{\rho,\rho'}\leq k\}}(\urho - \urhop)^2 M_{\rho,\rho'}^{\beta-1} \phi\, dx dt = 0, \qquad k\in\N .
\end{equation}
Fix $\epsilon\in (0,1)$ arbitrary. Let us consider
\begin{align*}
&\II{(\urho - \urhop)^2\phi} = 
\iint_{\{M_{\rho,\rho'}>k\}}(\urho - \urhop)^2\phi\, dx dt\\
&\quad + \iint_{\{M_{\rho,\rho'}\leq \epsilon\}}(\urho - \urhop)^2\phi\, dx dt +
\iint_{\{\epsilon< M_{\rho,\rho'}\leq k\}}(\urho - \urhop)^2\phi\, dx dt\\
&\leq 4\iint_{\{\urho >k\}} (\urho)^2\phi dx dt +2\epsilon^2\II{\phi}\\
&\quad  + 
\frac{1}{\epsilon^{\beta-1}}\iint_{\{\epsilon< M_{\rho,\rho'}\leq k\}} (\urho - \urhop)^2 M_{\rho,\rho'}^{\beta-1}\phi dx dt \\
&\leq 4 k^{1-\beta}\iint_{\{\urho >k\}} (\urho)^{\beta+1}\phi dx dt +2\epsilon^2\II{\phi}\\
&\quad  + 
\frac{1}{\epsilon^{\beta-1}}\iint_{\{\epsilon< M_{\rho,\rho'}\leq k\}} (\urho - \urhop)^2 M_{\rho,\rho'}^{\beta-1}\phi dx dt .
\end{align*}
From \eqref{bug.7} and the uniform bound for $\urho$ in $L^{\beta+1}(\QT)$ we deduce
\begin{align*}
\lim_{\rho,\rho'\to 0}
\II{(\urho - \urhop)^2\phi}\leq C(\epsilon^2 + k^{1-\beta}) .
\end{align*}
Since the left-hand side of the above inequality does not depend on $\epsilon$, $k$, we conclude
\begin{equation}\label{bug.6c}
\lim_{\rho,\rho'\to 0}\II{(\urho - \urhop)^2\phi} = 0.
\end{equation}
By choosing $\phi\in C^0_c(\QT)$, $\phi\geq 0$ such that $\phi\equiv 1$ on $Q_R\equiv B_R\times (R^{-1},T-R^{-1})$ for $R>2/T$ arbitrary (where $B_R$ is the ball of $\R^2$ with center 0 and radius $R$) we conclude from \eqref{bug.6c} that $\urho$ is a Cauchy sequence in $L^2(Q_R)$ (and therefore strongly convergent in such space) for every $R>2/T$. In particular, for every $R>0$ there exists a subsequence $u^{(\rho,R)}$ of $\urho$ that is a.e.~convergent in $Q_R$. A Cantor diagonal argument yields the existence of a subsequence (not relabeled) of $\urho$ that is a.e.~convergent in $Q_R$ for every $R\in\N$, and therefore $\urho\to u$ a.e.~in $\QT$. 

The a.e.~convergence of $\urho$ and the boundedness of $\urho$ in $L^{\beta+1}(\QT)$ imply
\begin{align}
\label{lim.ubeta}
(\urho)^\beta\rightharpoonup u^\beta\quad\mbox{weakly in }L^{\frac{\beta+1}{\beta}}(\QT).
\end{align}
Finally, since $\urho$ is bounded in $L^{\beta+1}(\QT)$ (see \eqref{vac.b}),
while $\na\prho$, $\rho^{\beta/(\beta+1)}\na(\urho)^{\beta-1}$ are bounded in
$L^\infty(0,T; L^2(\R^2))$, $L^{(\beta+1)/\beta}(\QT)$ (from \eqref{apr11.ei}), 
we deduce
\begin{align*}
&\left|\int_0^T\langle\pa_t\urho , \phi\rangle dt\right|\\ 
&\leq
\int_0^T\int_{\R^2}\urho |\na\prho| |\na\phi|dx dt + 
\rho\int_0^T\int_{\R^2}|\na(\urho)^{\beta-1}|^{1/\beta}|\na\phi| dx dt\\
&\leq \|\urho\|_{L^{\beta+1}(\QT)}
\|\na\prho\|_{L^\infty(0,T; L^2(\R^2))} 
\|\na\phi\|_{L^{(\beta+1)/\beta}(0,T; L^{2(\beta+1)/(\beta-1)}(\R^2))}\\
&\qquad + \rho\|\na(\urho)^{\beta-1}\|^{1/\beta}_{L^{(\beta+1)/\beta}(\QT)}
\|\na\phi\|_{L^{(\beta+1)/\beta}(\QT)}\\
&\leq C\|\phi\|_{L^{(\beta+1)/\beta}(0,T; W^{1,(\beta+1)/\beta}\cap W^{1,2(\beta+1)/(\beta-1)}(\R^2))}.
\end{align*}
As a consequence
\begin{align*}
\|\pa_t\urho\|_{L^{\beta+1}(0,T; (W^{1,(\beta+1)/\beta}\cap W^{1,2(\beta+1)/(\beta-1)}(\R^2))')}\leq C,
\end{align*}
and so
\begin{align}\label{lim.ut}
\pa_t\urho\rightharpoonup \pa_t u \quad\mbox{weakly in }L^{\beta+1}(0,T; (W^{1,(\beta+1)/\beta}\cap W^{1,2(\beta+1)/(\beta-1)}(\R^2))').
\end{align}
Putting the previous limit relations together allow us to take the limit $\rho\to 0$ inside \eqref{apr11.eq.1}, \eqref{apr11.eq.2} and obtain a solution to \eqref{weak.u}, \eqref{weak.p} (after a suitable density argument).
Finally, we show the mass conservation property. Define the cutoff
\begin{align*}
\eta_R(x) = 
\begin{cases}
1 & |x|<R\\
0 & |x|>2R\\
\frac{1}{2}(\cos\pi(|x|/R -1)+1) & R\leq |x|\leq 2R
\end{cases},\quad R>0.
\end{align*}
Let $\psi\in C^1_c[0,T)$ arbitrary. Choosing $\phi(x,t)=\psi(t)\eta_R(x)$ inside \eqref{weak.u} yields
\begin{align*}
&\left|\int_0^T\int_{\R^2}u\eta_R(x)\psi'(t) dxdt + \psi(0)\int_{\R^2}u_{in}\eta_R dx\right| \\
&= 
\left|\int_{\R^2}u\na p\cdot\na\eta_R\psi dxdt\right|\\
&\leq
\|u\|_{L^{\beta+1}(\QT)}\|\na p\|_{L^\infty(0,T; L^2(\R^2))}
\|\na\eta_R\|_{L^{\frac{2(\beta+1)}{\beta-1}}(\R^2)}
\|\psi\|_{L^{\frac{\beta+1}{\beta}}(0,T)}.
\end{align*}
Since $u\in L^{\beta+1}(\QT)$ and $\na p\in L^\infty(0,T; L^2(\R^2))$,
it follows
\begin{align}\label{uffa}
&\left|\int_0^T\int_{\R^2}u\eta_R(x)\psi'(t) dxdt + \psi(0)\int_{\R^2}u_{in}\eta_R dx\right| \\
\nonumber
&\qquad\leq C\|\na\eta_R\|_{L^{2+\delta}(\R^2)}\|\psi\|_{
	L^{\frac{\beta+1}{\beta}}(0,T)}\leq
C R^{-\frac{\delta}{2+\delta}}\|\psi\|_{
	L^{\frac{\beta+1}{\beta}}(0,T)}
\end{align}
with $\delta=\frac{4}{\beta-1}>0$. Choosing $\psi'\leq 0$ in $[0,T]$, taking the limit $R\to\infty$ inside \eqref{uffa} and applying the monotone convergence theorem yields 
$u\in L^\infty_{loc}(0,\infty ; L^1(\R^2))$ (since $T>0$ is arbitrary). At this point we can apply the dominated convergence theorem to take the limit $R\to\infty$ inside \eqref{uffa} with $\psi\in C^1_c([0,T))$ arbitrary and deduce
$$
-\int_0^T\int_{\R^2}u dx \,\psi'(t) dt = \psi(0)\int_{\R^2}u_{in} dx , \quad t>0,
$$
implying that the mass $\int_{\R^2}u dx$ is constant in time.
This concludes the proof of Theorem \ref{main_thm}.


\begin{thebibliography}{15}


 \bibitem{AS08} L. Ambrosio and S. Serfaty. {\em A gradient flow approach to an evolution problem arising in superconductivity. } Comm. Pure Appl. Math. 61 (2008), no. 11, 1495-1539. 
 
  \bibitem{BIK15} P. Biler, C.Imbert, and G. Karch.{\em  The nonlocal porous medium equation: Barenblatt profiles and other weak solutions.} Arch. Ration. Mech. Anal. 215 (2015), no. 2, 497 -529. 


\bibitem{CV15}
 L. Caffarelli and J.L. Vazquez.{\em  Regularity of solutions of the fractional porous medium flow with exponent 1/2. }Algebra i Analiz 27 (2015), no. 3, 125 - 156; translation in St. Petersburg Math. J. 27 (2016), no. 3, 437- 460.

\bibitem{CSV13}
L. Caffarelli, F. Soria and J.L. Vazquez. {\em Regularity of solutions of the fractional porous medium flow.} J. Eur. Math. Soc. (JEMS) 15 (2013), no. 5, 1701 -1746. 

\bibitem{CV11}
 L. Caffarelli, J.L. Vazquez. {\em Nonlinear porous medium flow with fractional potential pressure.} Arch. Ration. Mech. Anal. 202 (2011), no. 2, 537- 565. 
 
\bibitem{CheJueLiu14} X.~Chen, A.~J\uml ungel, J.-G.~Liu. {\em A note on Aubin-Lions-Dubinskiĭ lemmas.} Acta applicandae mathematicae 133.1 (2014): 33-43.
 
\bibitem{DGZ19} E. Daus, M. Gualdani and N. Zamponi.
{\em Long time behavior and weak-strong uniqueness for a nonlocal porous media equation.} 
J. Differential Equations 268 (2020), no. 4, 1820-1839. 

\bibitem{FeiNov} E.~Feireisl, A.~Novotn\'y. Singular Limits in Thermodynamics of Viscous Fluids. Basel: Birkh\uml auser, 2009.

\bibitem{GasMar08} I.~Gasser, P.~Marcati. On a generalization of the Div-Curl Lemma.
Osaka J.~Math.~45 (2008), 211-214.

\bibitem{GL97} G. Giacomin, J. Lebowitz.{\em Phase segregation dynamics in particle systems with long range interactions. I. Macroscopic limits.}  J. Statist. Phys. 87 (1997), no. 1-2, 37 -61. 

\bibitem{GL98} G. Giacomin, J. Lebowitz.{\em  Phase segregation dynamics in particle systems with long range interactions. II. Interface motion. }SIAM J. Appl. Math. 58 (1998), no. 6, 1707-1729. 

 \bibitem{GLM00} G. Giacomin, J. Lebowitz and R. Marra. {\em Macroscopic evolution of particle systems with short- and long-range interactions.} Nonlinearity 13 (2000), no. 6, 2143-2162.
 


\bibitem{LMG01} 
P.L. Lions, S. Mas-Gallic
{\em Une méthode particulaire déterministe pour des équations diffusives non linéaires. }
C. R. Acad. Sci. Paris Sér. I Math. 332 (2001), no. 4, 369-376. 


\bibitem{R91} F. Rezakhanlou. {\em Hydrodynamic limit for attractive particle systems on Zd.} Comm. Math. Phys. 140 (1991), no. 3, 417-448. 

\bibitem{STV18} D. Stan, F. del Teso, J. L. Vazquez. {\em Existence of weak solutions for a general porous medium equation with nonlocal pressure.} Arch. Rat. Mechanics Anal., 233, no 1 (2019), 451-496. 


\bibitem{SV14}
S. Serfaty and J.L. Vazquez. {\em A mean field equation as limit of nonlinear diffusions with fractional Laplacian operators. } Calc. Var. Partial Differential Equations 49 (2014), no. 3-4, 1091 - 1120. 


\bibitem{V18} J.L. Vazquez. {\em Asymptotic behaviour for the Fractional Heat Equation in the Euclidean space.} Complex Variables and Elliptic Equations, Special volume in honor of Vladimir I. Smirnov's 130th anniversary, 63 (2018), no. 7-8, 1216 - 1231.

\bibitem{Zei90}
E. Zeidler. {\em Nonlinear functional analysis and its applications}, vol. II/B (1990).

\end{thebibliography}
\end{document}